\documentclass{article}
\usepackage{amssymb,amsmath,amsthm,graphicx,cite}
\usepackage{color}

\def\qed{\hfill {\hbox{${\vcenter{\vbox{               
   \hrule height 0.4pt\hbox{\vrule width 0.4pt height 6pt
   \kern5pt\vrule width 0.4pt}\hrule height 0.4pt}}}$}}}

\def\utr{\, \underline{\triangleright}\, }
\def\otr{\, \overline{\triangleright}\, }
\def\ud{\, \underline{\bullet}\, }
\def\od{\, \overline{\bullet}\, }

\newtheorem{theorem}{Theorem}
\newtheorem{lemma}[theorem]{Lemma}
\newtheorem{proposition}[theorem]{Proposition}
\newtheorem{corollary}[theorem]{Corollary}
\newtheorem{conjecture}{Conjecture}
\newtheorem{observation}{Observation}
\theoremstyle{definition}
\newtheorem{example}{Example}
\newtheorem{definition}{Definition}
\newtheorem{remark}{Remark}

\date{}

\title{\Large \textbf{Psyquandles, Singular Knots and Pseudoknots}}

\author{Sam Nelson\footnote{Email: Sam.Nelson@cmc.edu. Partially supported by Simons Foundation collaboration grant 316709.}\and
Natsumi Oyamaguchi\footnote{Email: natsumi.3-29.math@diary.ocn.ne.jp}\and
Radmila Sazdanovic\footnote{Email: rsazdanovic@math.ncsu.edu. Partially supported by Simon Foundation collaboration grant 318086.}}

\begin{document}
\maketitle

\begin{abstract}
We generalize the notion of biquandles to \textit{psyquandles} and use these 
to define invariants of oriented singular links and pseudolinks. In addition to 
psyquandle counting invariants, we introduce Alexander psyquandles and 
corresponding invariants such as Alexander psyquandle polynomials and 
Alexander-Gr\"obner psyquandle invariants of oriented singular knots and links.
We consider the relationship between Alexander psyquandle 
colorings of pseudolinks and $p$-colorings of pseudolinks.
As a special case we define a generalization of the Alexander polynomial for
oriented singular links and pseudolinks we call the \textit{Jablan polynomial} 
and compute the invariant for all pseudoknots with up to five crossings and 
all 2-bouquet graphs with up to 6 classical crossings.
\end{abstract}

\parbox{6in} {\textsc{Keywords:} Biquandles, singular knots and links, spatial graphs, 2-bouquet graphs, pseudoknots, psyquandles, counting invariants, Alexander-Gr\"obner invariants, Jablan polynomial

\smallskip

}

\section{\large\textbf{Introduction}}\label{I}

First suggested in the mid 1990s \cite{FRS} and later developed in the 2000s
\cite{KR, FJK, EN}, \textit{biquandles} are algebraic structures whose axioms
are motivated by the oriented Reidemeister moves in knot theory. Biquandles
have been used since their introduction to define invariants of classical and 
virtual oriented knots and links. \cite{CES,EN,FJK,NOR, NR,NW,NV}

\textit{Singular knots and links} are 4-valent spatial graphs considered
up to \textit{rigid vertex isotopy}, where we may regard a vertex as the result
of two strands of a knot or link getting stuck together in a fixed position.
Singular knots and links are important in the study of \textit{Vassiliev 
invariants}; see \cite{GPV,V,V2}. In particular, a singular knot or link
with exactly one singular crossing is a \textit{2-bouquet graph}.

\textit{Pseudoknots} are knots including some \textit{precrossings}, classical 
crossings where we can't tell which strand goes on top. This definition, 
statistical in nature, is motivated by applications in molecular biology, such 
as modeling knotted DNA, where data often comes inconclusive with respect to 
which crossing it represents. 
\cite{HHJJMR,HJ,HJJ}.
\[\includegraphics{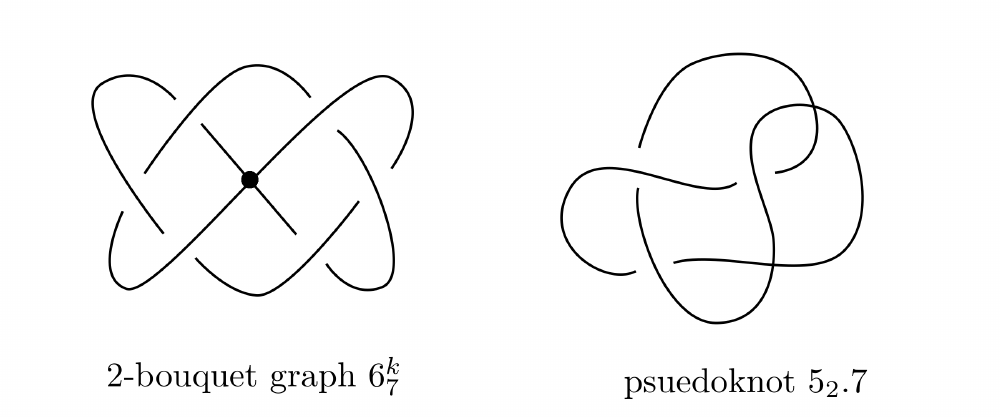}\]

Motivated by effectiveness of biquandles in distinguishing oriented knots and 
links, we introduce \textit{psyquandles} and use them to
define invariants of oriented singular knots and links and
oriented pseudoknots and pseudolinks. 
A psyquandle is a biquandle with additional structure 
in the form of operations at singular crossings or precrossings.
The paper is organized as follows. In Section \ref{SB} we review the
basic combinatorics of oriented singular knots and links and pseudoknots and
pseudolinks. In Section \ref{PSY} we introduce psyquandles and prove that
psyquandle colorings of singular knots and links and of pseudoknots and 
pseudolinks define invariants. In Section \ref{A} we introduce a particular 
type of psyquandle we call \textit{Alexander psyquandles} and use these to 
define analogs of the Alexander polynomials and Alexander-Gr\"obner invariants
for oriented singular knots and links and for oriented pseudoknots and 
pseudolinks. We consider the relationship between Alexander psyquandle 
colorings and $p$-colorings of pseudolinks as defined in \cite{HHJJMR}.
We introduce the \textit{Jablan polynomial} which generalizes the Alexander 
polynomial to the case of pseudolinks and singular links.
We end in Section \ref{Q} with some questions for future work.

\section{\large\textbf{Singular Knots and Pseudoknots}}\label{SB}

\textit{Singular knots and links} are rigid vertex isotopy classes of 4-valent 
spatial 
graphs. That is, a singular link diagram has classical crossings and 4-valent
vertices which are required to maintain a fixed cyclic ordering around the 
vertices. Geometrically, we can think of singular links as links with 
transverse self-intersections, each of which is fixed inside a small 
neighborhood.  An \textit{oriented singular knot or link} has oriented 
strands which pass through at each crossing and vertex; that is, the 
orientations are as pictured below.
\[\includegraphics{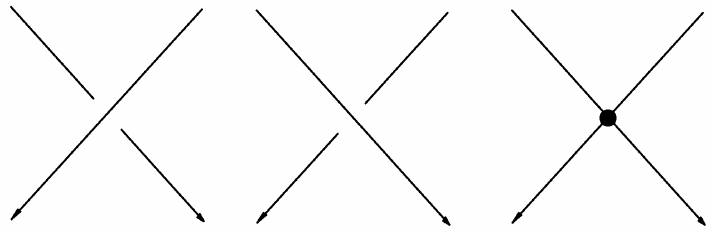}\]
Singular knot theory finds applications in the study of \textit{Vassiliev 
invariants}, integer-valued invariants of singular knots and links which
satisfy the \textit{Vassiliev skein relation}:
\[\includegraphics{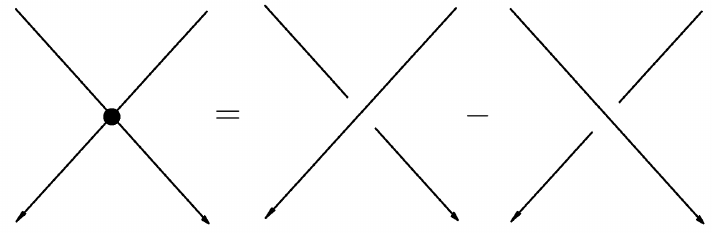}\]
See \cite{GPV,V,V2} for more.

\begin{example}
A \textit{2-bouquet graph} is a singular knot with exactly one singular 
crossing. 2-Bouquet graphs come in two types: \textit{K}-type 2-bouquet graphs 
form knots if the singular crossing is replaced with a classical crossing, while
\textit{L}-type 2-bouquet graphs form 2-component links when the singular 
crossing is replaced with a classical crossing.
The second listed author classified 2-bouquet graphs with up to six 
classical crossings in \cite{O}.
\end{example}

In \cite{BEHY} a generating set of three oriented singular moves is 
identified and shown to generate the remaining oriented singular moves:

\begin{theorem}\label{thm:behy} (BEHY)
In the presence of the oriented classical Reidemeister moves, the three moves
below generate the complete set of oriented singular moves.
\[\includegraphics{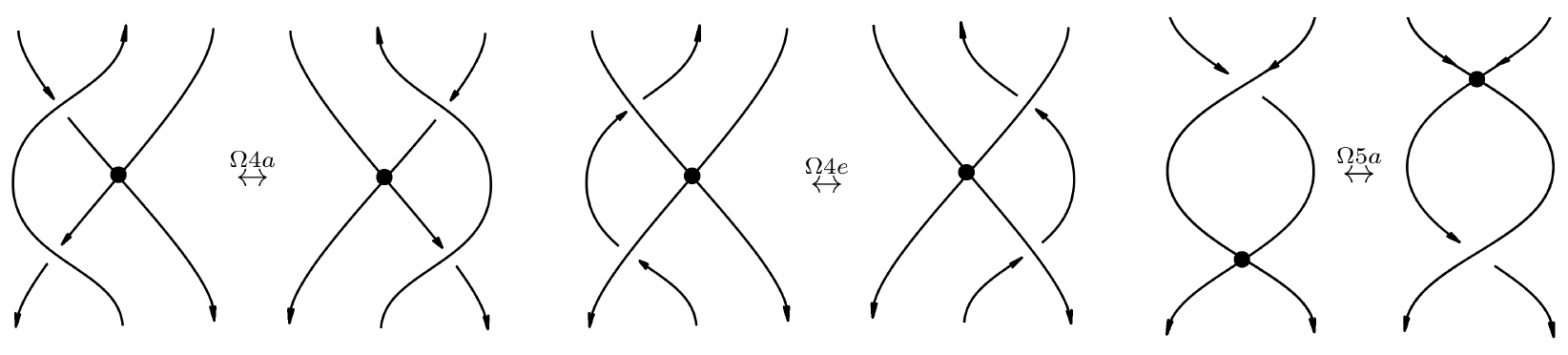}\]
\end{theorem}

For our purposes it will be easier to use an alternative generating set of 
singular moves.

\begin{proposition}\label{ourmoves}
In the presence of classical Reidemeister moves, the three moves
below generate the complete set of oriented singular moves.
\[\includegraphics{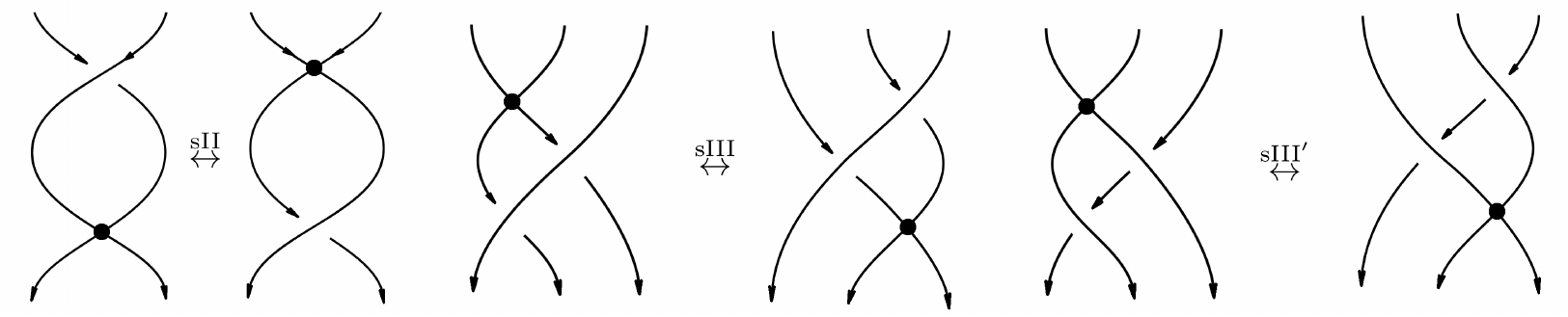}\]
\end{proposition}

\begin{proof}
It suffices to show that the moves in Theorem \ref{thm:behy} can be obtained 
using our preferred moves and the oriented classical Reidemeister moves.
Since move sII is the same as move $\Omega 5a$, we need only to show that 
moves $\Omega 4a$ and $\Omega 4e$ can be obtained using the classical 
Reidemeister moves and moves sII, sIII and sIII$'$. Then consider the case of
$\Omega 4a$; we will obtain it using sIII and two classical Reidemeister II 
moves.  
\[\includegraphics{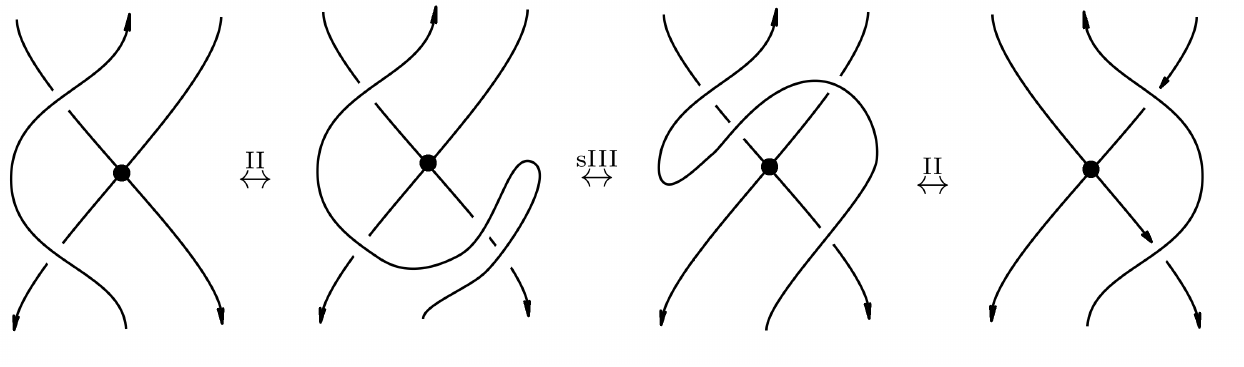}\]
The case of $\mathrm{sIII}'\Rightarrow \Omega4e$ is similar.
\end{proof}

\textit{Pseudoknots} are knots and links which in addition to classical
crossings include some \textit{precrossings}, classical crossings in which 
it is unknown which strand goes over and which strand goes under. While the 
concept originated in biology where limited resolution in pictures of knotted
molecules makes it difficult to tell which strand in on top, the current 
mathematical study of pseudoknots 
was initiated in \cite{H} and continued in papers such as \cite{HHJJMR,HJ,HJJ}.
A precrossing is drawn as an undecorated self-intersection:
\[\includegraphics{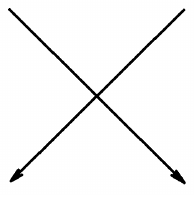}\]

The Reidemeister moves for pseudoknots (see \cite{HJJ} etc.)
are, conveniently, very similar to our preferred set of Reidemeister moves for
oriented singular knots:
\[\includegraphics{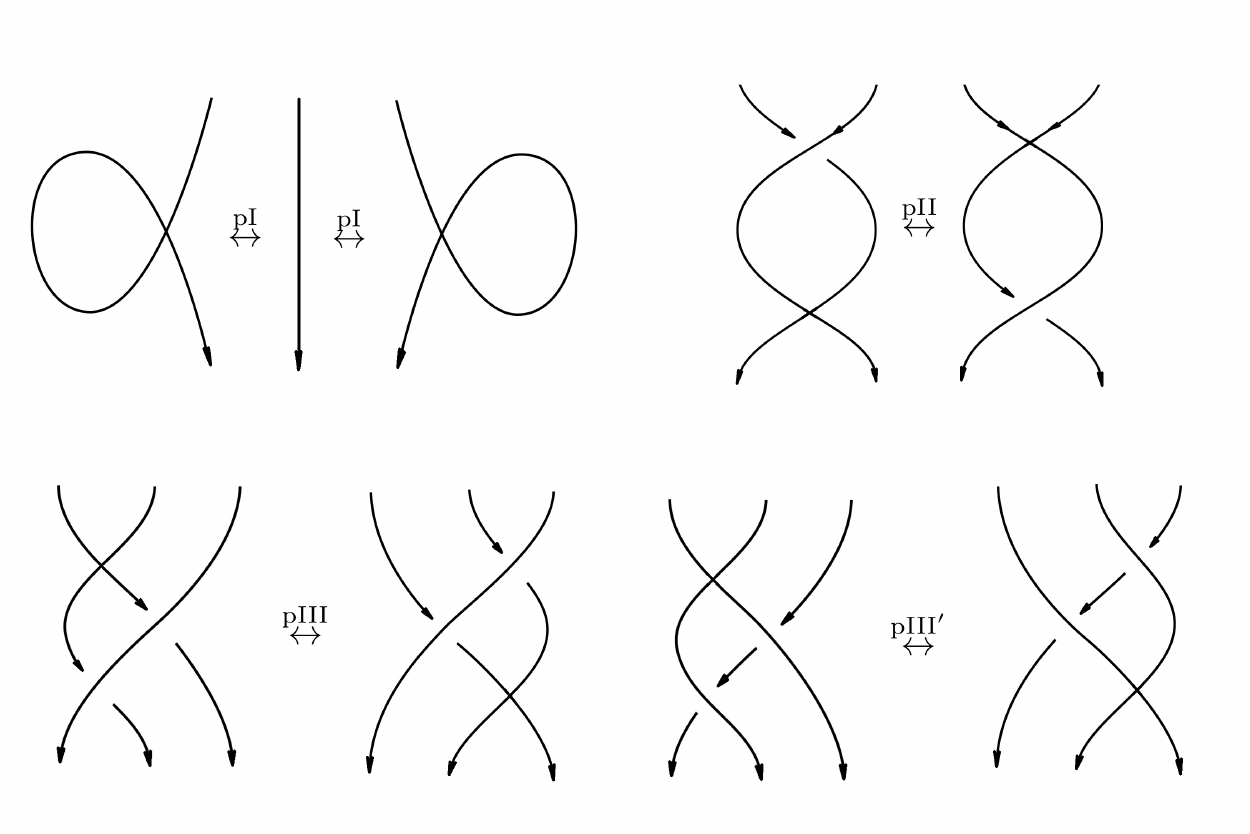}\]
Indeed, after replacing singular crossings with precrossings, the only 
difference is the addition of a Reidemeister I-style 
move with a precrossing, no analog of which exists for singular knots.

A \textit{resolution} of a pseudolink diagram is an assignment of classical 
crossing type to each of the precrossings in the diagram. A powerful invariant
of pseudolinks is the \textit{weighted resolution set} or \textit{WeRe set}, the 
discrete probability distribution consisting of the set of resolution link 
types and their associated probabilities with the assumption that both crossing 
resolutions are equally probable.

\begin{example}\label{ex:pl}
The pseudolink below has the listed WeRe set where $0_2$ is the unlink of two 
components and L2a1 is the Hopf link.
\[\raisebox{-0.75in}{\includegraphics{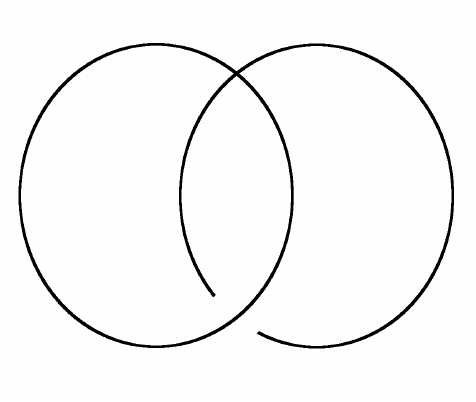}}\quad
\left\{\left(0_2,\frac{1}{2}\right),\left(L2a1,\frac{1}{2}\right)\right\}\]
\end{example}

\section{\Large\textbf{Psyquandles}}\label{PSY}

The similarity of the singular Reidemeister moves with the pseudoknot
Reidemeister moves suggests introducing a single algebraic structure
for coloring these objects with new operations at the singular crossings
or precrossings.

Recall (see \cite{EN} for example) that a \textit{biquandle} is a set $X$
with operations $\utr,\otr:X\times X\to X$ satisfying
\begin{itemize}
\item[(i)] For all $x\in X$, $x\utr x=x\otr x$,
\item[(ii)] For all $x,y\in X$, the maps $\alpha_y,\beta,y:X\to X$ and 
$S:X\times X\to X\times X$ defined by
\[\alpha_y(x)=x\otr y, \quad \beta_y(x)=x\utr y\quad \mathrm{and}\quad
S(x,y)=(S_1(x,y),S_2(x,y))=(y\otr x,x\utr y)\]
are invertible, and
\item[(iii)] For all $x,y,z\in X$ the \textit{exchange laws} are satisfied:
\[\begin{array}{rcl}
(x\utr y)\utr(z\utr y) & = & (x \utr z)\utr (y\otr z) \\
(x\utr y)\otr(z\utr y) & = & (x \otr z)\utr (y\otr z) \\
(x\otr y)\otr(z\otr y) & = & (x \otr z)\otr (y\utr z). \\
\end{array}\]
\end{itemize}
The biquandle axioms are motivated by the classical Reidemeister moves where
we label the \textit{semiarcs} in a knot diagram (the edges in the graph 
obtained from the diagram by making each crossing a 4-valent vertex) as shown:
\[\includegraphics{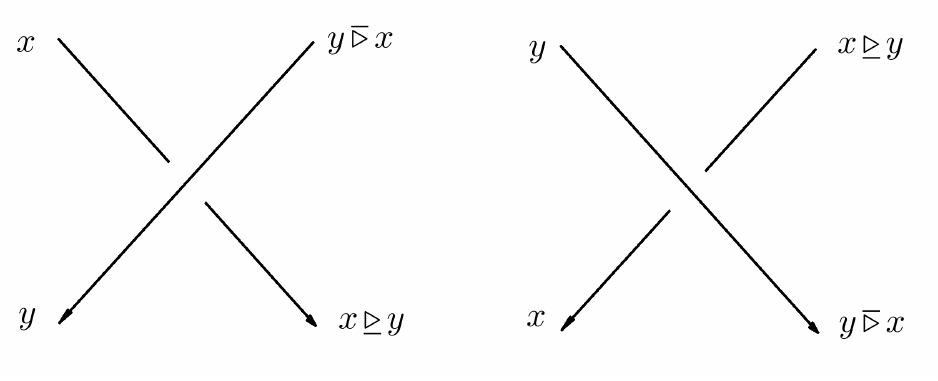}\] 
Axiom (ii) is equivalent to the \textit{adjacent labels rule}, which says that
the colors of any two adjacent semiarcs determine the colors of the other two.

\begin{definition}
A \textit{psyquandle} is a biquandle $X$ with two additional binary operations 
$\ud,\od:X\times X\to X$ satisfying the conditions
\begin{itemize}
\item[(p/si)] For all $x,y\in X$, the maps $\alpha'_y,\beta'_y:X\to X$ and 
$S':X\times X\to X\times X$ defined by
\[\alpha'_y(x)=x\od y, \quad\beta'_y(x)=x\ud y\quad \mathrm{and}\quad
S'(x,y)=(S_1'(x,y),S_2'(x,y))=(y\od x,x\ud y)\]
are invertible,
\item[(p/sii)] For all $x,y\in X$ there exist unique $w,z\in X$ such that
\[\begin{array}{rcl}
x\utr y & = & z\od y \\
y\otr x & = & w\od x \\
w\utr z & = & y\ud z \\
z\otr w & = & x\ud w \\
\end{array}
\]
and
\item[(p/siii)] For all $x,y,z\in X$ we have the \textit{mixed exchange laws}:
\[\begin{array}{rcl}
(x\otr y)\otr(z\od y) & = & (x\otr z)\otr(y\ud z)\\
(x\utr y)\utr(z\od y) & = & (x\utr z)\utr(y\ud z)\\
(x\otr y)\od(z\otr y) & = & (x\od z)\otr(y\utr z)\\
(x\utr y)\ud(z\utr y) & = & (x\ud z)\utr(y\otr z)\\
(x\otr y)\ud(z\otr y) & = & (x\ud z)\otr(y\utr z)\\
(x\utr y)\od(z\utr y) & = & (x\od z)\utr(y\otr z)\\
\end{array}\]
\end{itemize}
A psyquandle is \textit{pI-adequate} if it additionally satisfies for all
$x\in X$
\[x\ud x=x\od x.\]
\end{definition}

\begin{definition}\label{def:psy}
Let $X$ be a psyquandle (respectively, a pI-adequate psyquandle) and $L$ an 
oriented singular link (respectively, oriented pseudolink) diagram. Then an
\textit{$X$-coloring} of $L$ is an assignment of elements of $X$to the semiarcs
in $L$ such that every crossing we have the following:
\[\scalebox{0.85}{\includegraphics{sn-no-rs-2.pdf}\includegraphics{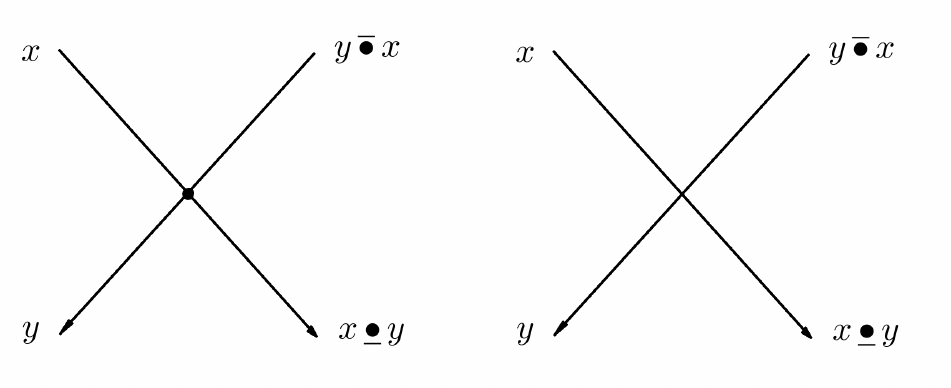}}\] 
\end{definition}

The psyquandle axioms are motivated by the moves $\{$p/sII, p/sIII, p/sIII$'\}$
(and in the case of pI-adequate psyquandles, move pI) using the coloring rule
in definition \ref{def:psy}
at singular crossings and precrossings. In particular, we have:

\begin{theorem}
Let $L$ be an oriented singular link (respectively, pseudolink) diagram.
For any finite psyquandle $X$, the number of $X$-colorings of $L$ is not
changed by Reidemeister moves and hence defines an invariant 
$\Phi_X^{\mathbb{Z}}(L)$ called the \textit{psyquandle counting invariant.}
\end{theorem}

\begin{proof}
We verify for each of the moves pI, p/sII, p/sIII and p/sIII$'$. First, move
pI requires $x\ud x=x\od x$ for all $x\in X$:
\[\includegraphics{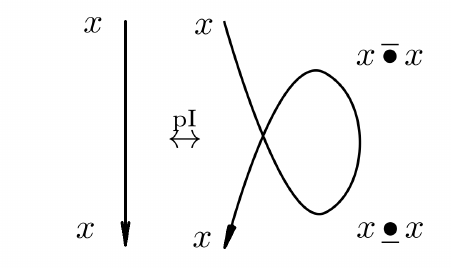}
\includegraphics{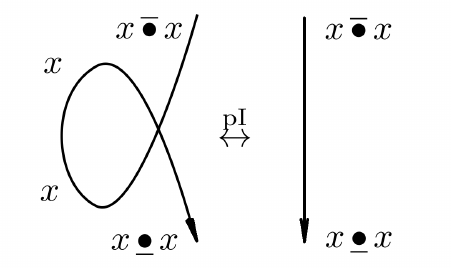}\]
Next, let us consider the p/sII move. 
\[\includegraphics{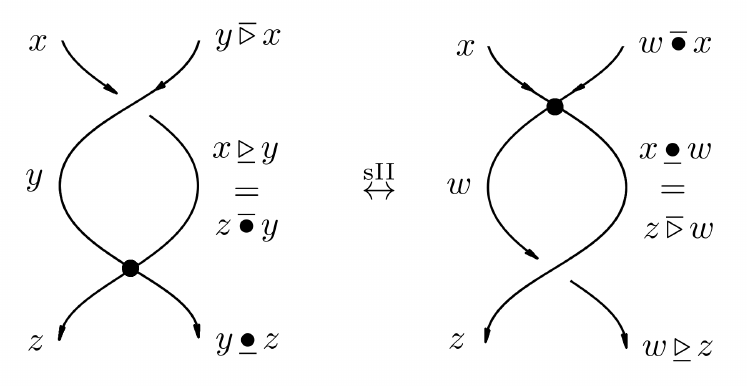}\]  
We want each $X$-coloring of the
diagram on the left to correspond to exactly one $X$-coloring of the
diagram on the right. The fact that $\utr,\otr,\ud$ and $\od$ satisfy the
adjacent labels rule implies that the colors $x,y$ determine all the 
semiarc colors in the left diagram and the requirement that colors agree on the
boundary of the neighborhood of the move the implies that $x,y$ also determine 
the colors in the right diagram. Then for each pair $x,y\in X$ 
there should be unique $z,w$ satisfying the pictured conditions.

Finally, for the p/sIII and p/sIII$'$ moves, we compare semiarc labels
on both sides of the moves.
\[\includegraphics{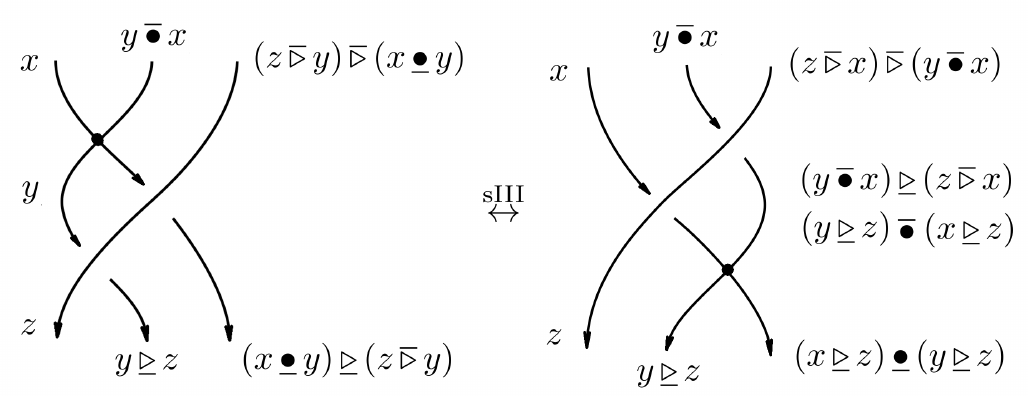}\] 
\[\includegraphics{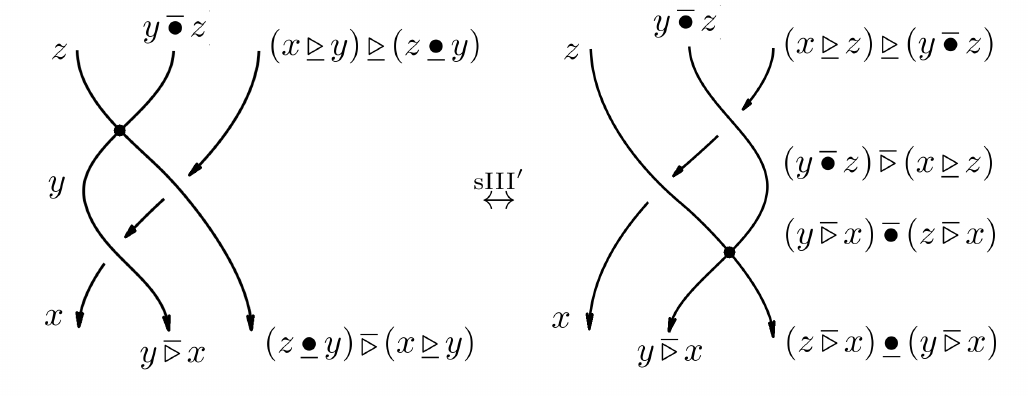}\] 
\end{proof}

\begin{example}
Let $X$ be a biquandle. Replacing the singular/precrossing with a positive 
crossing shows that that setting $\od=\otr$ and $\ud=\utr$ yields a pI-adequate
psyquandle, and replacing it with a negative crossing shows that setting
$\od=\utr$ and $\ud=\otr$ yields a pI-adequate psyquandle.
\end{example}

\begin{definition}
A \textbf{pure psyquandle} is a psyquandle with trivial classical operations,
i.e. a psyquandle $X$ such that
\[x\utr y=x\otr y =x\]
for all $x,y\in X$. We note that the mixed exchange laws are automatically 
satisfied in this case, so every pair of operations $x\ud y, x\od y$
satisfying (p/si) and (p/sii) is a pure psyquandle.
\end{definition}

\begin{example}
Let $X$ be a set and $\sigma,\tau:X\to X$ bijections. Then 
$x\utr y=x\otr y=\tau(x)$ defines a biquandle operation called a 
\textit{constant action biquandle}. Defining $x\ud y=x\od y=\sigma(x)$
makes this a pI-adequate psyquandle we call a \textit{constant action 
psyquandle} provided 
\[\sigma^{-1}\tau=\tau^{-1}\sigma \quad\mathrm{and}\quad \sigma\tau=\tau\sigma.\]
We verify the axioms:
\begin{itemize}
\item[(p/si)] $\alpha'=\beta'=\sigma$ is invertible and 
$S'^{-1}(x,y)=(\tau^{-1}(y),\tau^{-1}(x))$,
\item[(p/sii)] Given $x,y\in X$, define $z=\tau^{-1}\sigma(x)$ and 
$w=\tau^{-1}\sigma(y)$. Then we have
\[\begin{array}{rcccccl}
x\utr y &= & \sigma(x) & = & \tau(\tau^{-1}(\sigma(x))) & =& z\otr y \\
y\otr x &= & \sigma(y) & = & \tau(\tau^{-1}(\sigma(y))) & =& w\otr x \\
w\utr z &= & \sigma(\tau^{-1}(\sigma(y))) & = & \tau(y) & =& y\ud z \\
z\otr w &= & \sigma(\tau^{-1}(\sigma(x))) & = & \tau(x) & =& x\od w 
\end{array}\]
and
\item[(p/siii)] For all $x,y,z\in X$ we have
\[
\begin{array}{rcccccl}
(x\otr y)\otr(z\od y) & = & \sigma^2(x) & = & \sigma^2(x) & = & (x\otr z)\otr(y\ud z)\\
(x\utr y)\utr(z\od y) & = & \sigma^2(x) & = & \sigma^2(x) & = & (x\utr z)\utr(y\ud z)\\
(x\otr y)\od(z\otr y) & = & \sigma(\tau(x)) &= & \tau(\sigma(x)) & = & (x\od z)\otr(y\utr z)\\
(x\utr y)\ud(z\utr y) & = & \sigma(\tau(x)) &= & \tau(\sigma(x)) & = & (x\ud z)\utr(y\otr z)\\
(x\otr y)\ud(z\otr y) & = & \sigma(\tau(x)) &= & \tau(\sigma(x)) & = & (x\ud z)\otr(y\utr z)\\
(x\utr y)\od(z\utr y) & = & \sigma(\tau(x)) &= & \tau(\sigma(x)) & = & (x\od z)\utr(y\otr z)\\
\end{array}\]
as required.
\end{itemize}
\end{example}

\begin{example}
We can express a psyquandle structure on a finite set $X=\{x_1,\dots, x_n\}$ 
with an $n\times 4n$ matrix encoding the operation tables of 
$\utr,\otr,\ud,\od$ where the $(j,k)$ entry $m$ in the matrix satisfies
\[x_m=\left\{\begin{array}{ll}
x_j\utr x_k & 1\le k\le n \\
x_j\otr x_k & n+1\le k\le 2n \\
x_j\ud x_k & 2n+1\le k\le 3n \\
x_j\od x_k & 3n+1\le k\le 4n \\
\end{array}\right.\]
For instance, the constant action psyquandle on $X=\{x_1,x_2,x_3,x_4\}$ 
where $\sigma=(12)$ and $\tau=(34)$ has operation matrix
\[\left[\begin{array}{rrrr|rrrr|rrrr|rrrr}
2 & 2 & 2 & 2 & 2 & 2 & 2 & 2 & 1 & 1 & 1 & 1 & 1 & 1 & 1 & 1 \\
1 & 1 & 1 & 1 & 1 & 1 & 1 & 1 & 2 & 2 & 2 & 2 & 2 & 2 & 2 & 2 \\
3 & 3 & 3 & 3 & 3 & 3 & 3 & 3 & 4 & 4 & 4 & 4 & 4 & 4 & 4 & 4 \\
4 & 4 & 4 & 4 & 4 & 4 & 4 & 4 & 3 & 3 & 3 & 3 & 3 & 3 & 3 & 3 \\
\end{array}\right]\]
\end{example}

\begin{example}
Let $X=\{1,2,3\}$. The operation matrix
\[\left[\begin{array}{rrr|rrr|rrr|rrr}
1 & 1 & 1 & 1 & 1 & 1 & 2 & 2 & 2 & 3 & 2 & 2 \\
2 & 2 & 2 & 2 & 2 & 2 & 1 & 3 & 1 & 1 & 1 & 1\\
3 & 3 & 3 & 3 & 3 & 3 & 3 & 1 & 3 & 2 & 3 & 3
\end{array}\right]\]
defines a pure psyquandle which is not pI-adequate.
\end{example}

\begin{definition}
Let $D$ be an oriented singular or pseudolink diagram representing an
oriented singular or pseudolink $L$ and let $G=\{g_1,\dots, g_n\}$ be a set 
of symbols corresponding to the semiarcs in $D$. We define the 
\textit{fundamental psyquandle} of $D$ is the usual universal algebraic way, 
namely:
\begin{itemize}
\item The set $W(D)$ of \textit{psyquandle words} in $G$ is defined 
recursively by the rules 
\begin{itemize}
\item[(i)] $G\subset W(G)$ and
\item[(ii)] $x,y\in W(G)$ implies \[ x\utr y, x\otr y, x\ud y, x\od y,
\alpha^{-1}_y(x),\alpha'^{-1}_y(x),\beta^{-1}_y(x),\beta'^{-1}_y(x),\] 
\[S_1(x,y), S_2(x,y), S_1'(x,y), S_2'(x,y), w(x,y), z(x,y) \in W(G),\] 
\end{itemize}
\item We make an equivalence relation on $W(G)$ generated by relations 
representing the psyquandle axioms, e.g. 
\[(x\otr y)\otr(z\od y)\sim (x\otr z)\otr(y\ud z),\ x\utr y\sim z(x,y)\od y, \mathrm{etc.,} \]
\item The \textit{free psyquandle on $G$} is the set of equivalence classes of
$W(G)$ modulo this equivalence relation; if we include axiom (pi) we obtain 
the \textit{free pI-adequate psyquandle}, and
\item Including the crossing relations form Definition \ref{def:psy} in our 
equivalence relation yields the \textit{fundamental psyquandle} of $D$,
denoted $\mathcal{P}(D)$ or $\mathcal{P}_I(D)$ for the fundamental 
pI-adequate psyquandle.
\end{itemize}
\end{definition}

\begin{theorem}
The isomorphism class $\mathcal{P}(L)$ of $\mathcal{P}(D)$ is an invariant of 
oriented singular links, and the isomorphism class $\mathcal{P}_I(L)$ of 
$\mathcal{P}_I(D)$ is an invariant of oriented pseudolinks.
\end{theorem}

\begin{proof}
By construction, Reidemeister moves on diagrams induce Tietze moves on 
presentations of $\mathcal{P}(D)$ and $\mathcal{P}_I(D)$ respectively, 
resulting in isomorphic psyquandles.
\end{proof}

Psyquandles form a category with psyquandles as objects and \textit{psyquandle
homomorphisms}, maps $f:X\to Y$ satisfying
\[f(x\utr y)=f(x)\utr f(y),\quad f(x\otr y)=f(x)\otr f(y),\quad
f(x\ud y)=f(x)\ud f(y)\quad \mathrm{and}\quad
f(x\od y)=f(x)\od f(y)\]
as morphisms.

Let $D$ be an oriented singular link or pseudolink diagram and let $X$ be a 
finite psyquandle. An assignment of elements of $X$ to the semiarcs in $D$
defines a homomorphism $f:\mathcal{P}(D)\to X$ if and only if the coloring
conditions in Definition \ref{def:psy} are satisfied at every crossing;
we will refer to such an assignment as an \textit{$X$-coloring} of $D$.
Thus, we can compute the the set of psyquandle homomorphisms 
$\mathrm{Hom}(\mathcal{P}(L),X)$ for an oriented singular link or pseudolink 
$L$ by computing the set of $X$-colorings of a diagram $D$ representing $L$.
More precisely, fixing an ordering of the semiarcs in $D$ gives us a way
to represent homomorphisms $f\in\mathrm{Hom}(\mathcal{P}(L),X)$ concretely
as ordered tuples of elements of $X$.
The number of such colorings is an integer-valued invariant of singular 
links and pseudolinks we call the \textit{psyquandle counting invariant},
denoted $\Phi_X^{\mathbb(Z)}(L)=|\mathrm{Hom}(\mathcal{P}(L),X)|$.

\begin{example}\label{ex:11l}
Consider the psyquandle $X$ with operation matrix
\[\left[\begin{array}{rr|rr|rr|rr}
2 & 2 & 2 & 2 & 2 & 2 & 2 & 2 \\
1 & 1 & 1 & 1 & 1 & 1 & 1 & 1 
\end{array}\right].
\]
The $2$-bouquet graph $1_1^l$ below 
\[\includegraphics{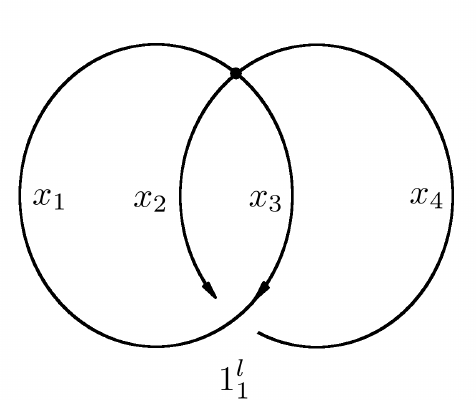}\]
has 4 $X$-colorings, each of which we can identify explicitly
as a $4$-tuple $(f(x_1),f(x_2), f(x_3), f(x_4))$:
\[\mathrm{Hom}(\mathcal{P}(1_1^l),X)=\{(1,1,2,2),(1,2,2,1),(2,1,2,1),(2,2,1,1)\}.\]
This distinguishes this link from the $2$-bouquet graph $0_1^k$
\[\includegraphics{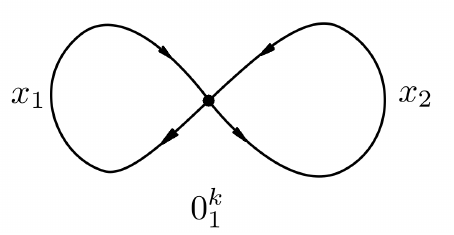}\]
which has only two $X$-colorings
\[\mathrm{Hom}(\mathcal{P}(0_1^k),X)=\{(1,2),(2,1)\}.\]
\end{example}

\begin{example}\label{ex:tbg}
Using our custom \texttt{Python} code, we computed the counting invariant 
for the 2-bouquet graphs (with choices of orientation) in \cite{O} using the 
psyquandle with operation matrix
\[\left[
\begin{array}{rrrrrr|rrrrrr|rrrrrr|rrrrrr}
2& 4& 4& 6& 6& 2& 2& 6& 2& 6& 2& 6& 2& 4& 2& 6& 2& 2& 2& 6& 4& 6& 6& 6 \\
3& 5& 5& 1& 1& 3& 1& 5& 1& 5& 1& 5& 3& 5& 5& 5& 1& 5& 1& 5& 1& 1& 1& 3 \\ 
4& 6& 6& 2& 2& 4& 6& 4& 6& 4& 6& 4& 6& 6& 6& 2& 6& 4& 4& 4& 6& 4& 2& 4 \\
5& 1& 1& 3& 3& 5& 5& 3& 5& 3& 5& 3& 5& 3& 1& 3& 3& 3& 5& 1& 5& 3& 5& 5 \\
6& 2& 2& 4& 4& 6& 4& 2& 4& 2& 4& 2& 4& 2& 4& 4& 4& 6& 6& 2& 2& 2& 4& 2 \\
1& 3& 3& 5& 5& 1& 3& 1& 3& 1& 3& 1& 1& 1& 3& 1& 5& 1& 3& 3& 3& 5& 3& 1 \\
\end{array}\right].\]
The results are collected in the table.
\[
\begin{array}{r|l}
\Phi_X^{\mathbb{Z}}(L) & L \\ \hline 
6 & 1_1^k, 3_1^k, 4_1^k, 4_2^k, 5_1^k, 5_4^k, 5_5^k, 5_6^k, 5_7^k, 5_8^k, 6_1^k, 6_2^k, 6_3^k, 6_4^k, 6_5^k, 6_6^k, 6_8^k, 6_9^k, 6_{10}^k, 6_{11}^k,6_{12}^k, 6_{13}^k, 6_{14}^k, 6_{15}^k, 6_{18}^k \\
8 & 5_3^l, 6_5^l \\
12 & 3_1^l,4_1^l,  5_2^l, 5_3^l, 6_1^l, 6_2^l, 6_6^l \\
18 & 2_1^k, 5_2^k, 5_3^k, 6_7^k, 6_{16}^k, 6_{17}^k, 6_{19}^k \\ 
24 & 1_1^l, 5_1^l, 6_3^l,6_5^l,  6_7^l, 6_8^l, 6_9^l, 6_{10}^l, 6_{11}^l \\
36 & 6_4^l, 6_{12}^l.
\end{array}
\]
\end{example}

\begin{example}
Noticing that the psyquandle in example \ref{ex:tbg} is pI-adequate since the
two right blocks have the same diagonal, we computed the counting invariant for
a choice of orientations for the pseudoknots in \cite{HHJJMR}. The results are 
collected in the table.
\[
\begin{array}{r|l}
\Phi_X^{\mathbb{Z}}(L) & L \\ \hline 
6 & 3_1.2, 3_1.3, 4_1.4, 4_1.3, 4_1.4, 4_1.5, 5_1.1, 5_1.3, 5_1.4, 5_2.1, 5_2.2,5_2.3,5_2.4, 5_2.5, 5_2.6, 5_2.7, 5_2.8, 5_2.9, 5_2.10 \\
18 & 3_1.1, 4_1.1, 5_1.2, 5_1.5.
\end{array}\]
\end{example}

\section{\Large\textbf{Alexander Psyquandles}}\label{A}

Let $\Lambda=\mathbb{Z}[t^{\pm 1},s^{\pm 1}]$. Any $\Lambda$-module 
$X$ is a biquandle under the operations
\[x\utr y=tx+(s-t)y\quad\mathrm{and}\quad x\otr y=sx\]
known as an \textit{Alexander biquandle}
(see \cite{EN} or \cite{KR}). Interpreting the fundamental biquandle of a knot
or link as an Alexander biquandle yields invariants including the Alexander 
polynomials and generalizations such as the Sawollek polynomials \cite{KR,S} 
and the Alexander-Gr\"obner invariants \cite{CHN}. In this section we will 
extend this definition to the case of psyquandles and as an application define 
notions of Alexander polynomials, Alexander-Gr\"obner invariants and a special
case we call \textit{Jablan polynomials} for singular and pseudoknots and links.

\begin{proposition}
Let $\Lambda'=\mathbb{Z}[t^{\pm 1}, s^{\pm 1}, a^{\pm 1},b^{\pm 1}]/(s+t-a-b)$ and 
let $X$ be a $\Lambda'$-module. The operations
\begin{eqnarray*}
x\utr y & = & tx+(s-t)y \\
x\otr y & = & sx \\
x\ud y & = &  ax+(s-a)y \\
x\od y & = &  bx+(s-b)y
\end{eqnarray*}
make $X$ a pI-adequate psyquandle called an \textit{Alexander psyquandle}.
\end{proposition}

\begin{proof}
We verify the axioms. First, checking pI-adequacy, we have
\[x\od x = bx+(s-b)x=sx=ax+(s-a)x=x\ud x.\]
Next, for axiom (p/si) if we define 
\begin{eqnarray*}
\alpha'_y(x) & = & bx+ (s-b)y \\
\beta'_y(x) & = & ax+(s-a)y \\
S'(x,y) & = & ((s-b)x+by,\ ax+(s-a)y)
\end{eqnarray*}
then setting
\begin{eqnarray*}
\alpha'^{-1}_y(x) & = & b^{-1}(x-(s-b)y) \\
\beta'^{-1}_y(x) & = & a^{-1}(x-(s-a)y) \\
S'^{-1}(x,y) & = & 
((s^{-1}-bs^{-1}t^{-1})x+bs^{-1}t^{-1}y,\ as^{-1}t^{-1}x+(s^{-1}-as^{-1}t^{-1})y)
\end{eqnarray*}
yields the inverse maps. Let us verify:
\begin{eqnarray*}
\alpha'_y(bx+(s-b)y) & = & b^{-1}(bx+(s-b)y-(s-b)y))= x \\
\beta'^{-1}_y(ax+(s-a)y) & = & a^{-1}(ax+(s-a)y-(s-a)y) =x \\
\end{eqnarray*}
and writing 
\[S'^{-1}((s-b)x+by,\ ax+(s-a)y) = (Ax+By,Cx+Dy)\]
we compute
\begin{eqnarray*}
A & = & (s^{-1}-bs^{-1}t^{-1})(s-b)+abs^{-1}t^{-1} \\
 & = & 1-bs^{-1}-bt^{-1}+b^2s^{-1}t^{-1}+abs^{-1}t^{-1} \\
 & = & 1-bs^{-1}-bt^{-1}+(s+t-a)bs^{-1}t^{-1}+abs^{-1}t^{-1} \\
 & = & 1-bs^{-1}-bt^{-1}+bt^{-1}+bs^{-1}-abs^{-1}t^{-1}+abs^{-1}t^{-1} \\
 & = & 1 
\end{eqnarray*}\begin{eqnarray*}
B & = & (s^{-1}-bs^{-1}t^{-1})b+bs^{-1}t^{-1}(s-a)\\ 
 & = & bs^{-1}-b^2s^{-1}t^{-1}+bt^{-1}-abs^{-1}t^{-1} \\
 & = & bs^{-1}-(s+t-a)bs^{-1}t^{-1}+bt^{-1}-abs^{-1}t^{-1} \\
 & = & bs^{-1}-bt^{-1}+bs^{-1}+abs^{-1}t^{-1}+bt^{-1}-abs^{-1}t^{-1} \\
 &= & 0
\end{eqnarray*}\begin{eqnarray*}
C & = & as^{-1}t^{-1}(s-b)+(s^{-1}-as^{-1}t^{-1})a \\
& = & at^{-1}-abs^{-1}t^{-1}+(s^{-1}-(s+t-b)s^{-1}t^{-1})a \\
& = & at^{-1}-abs^{-1}t^{-1}+as^{-1}-at^{-1} -as^{-1}+abs^{-1}t^{-1} \\
& =& 0\quad \mathrm{and}
\end{eqnarray*}\begin{eqnarray*}
D & = &  abs^{-1}t^{-1}+(s^{-1}-as^{-1}t^{-1})(s-a) \\
& = & abs^{-1}t^{-1}+1-at^{-1}-as^{-1}+a^2s^{-1}t^{-1}) \\
& = & abs^{-1}t^{-1}+1-at^{-1}-as^{-1}+a(s+t-b)s^{-1}t^{-1}) \\
& = & abs^{-1}t^{-1}+1-at^{-1}-as^{-1}+at^{-1}+as^{-1}-abs^{-1}t^{-1} \\
& = & 1
\end{eqnarray*}
and axiom (p/si) is satisfied.

To verify axiom (p/sii), we observe that given $x,y$ we can define
\begin{eqnarray*}
w & = & b^{-1}(b-s)x +b^{-1}sy\\
z & = & b^{-1}tx + b^{-1}(s-a)y
\end{eqnarray*}
and then we have
\begin{eqnarray*}
bw+(s-b)x 
& = & b(b^{-1}(b-s)x+b^{-1}sy)+(s-b)x \\
& = & (b-s)x+(s-b)x+sy\\
& = & sy, 
\end{eqnarray*}\begin{eqnarray*}
tw+(s-t)z 
& = & t(b^{-1}(b-s)x +b^{-1}sy)+(s-t)(b^{-1}tx + b^{-1}(s-a)y)\\
& = & b^{-1}(tb-ts+ts-t^2)x+b^{-1}(st+s^2-st-as+at)y \\
& = & b^{-1}t(b-t)x+b^{-1}(s^2-as+at)y \\
& = & (s-a)b^{-1}tx + b^{-1}(s^2-2as+a^2 +as+at-a^2)y \\
& = & (s-a)b^{-1}tx + b^{-1}(s^2-2as+a^2 +a(s+t-a))y \\
& = & (s-a)b^{-1}tx + (b^{-1}(s-a)^2 +a)y \\
& = & ay+(s-a)(b^{-1}tx + b^{-1}(s-a)y) \\
& = & ay+(s-a)z, 
\end{eqnarray*}\begin{eqnarray*}
bz+(s-b)y 
& = & b(b^{-1}tx + b^{-1}(s-a)y)+(s-b)y \\
& = & tx + (s-a)y+(a-t)y \\
& = & tx+(s-t)y \quad\mathrm{and} 
\end{eqnarray*}\begin{eqnarray*}
ax+(s-a)w 
& = & ax+(s-a)(b^{-1}(b-s)x +b^{-1}sy) \\
& = & b^{-1}(ab+(s-a)(b-s))x +b^{-1}(s-a)sy \\
& = & b^{-1}(ab+sb-ab-s^2+as)x +b^{-1}(s-a)sy \\
& = & b^{-1}(s(b-s+a)x +b^{-1}(s-a)sy \\
& = & sb^{-1}tx + b^{-1}(s-a)y\\
& = & sz 
\end{eqnarray*}
as required.

Finally, for axiom (p/siii) we verify each of the mixed exchange laws:
\begin{eqnarray*}
(x\otr y)\otr(z\od y) & = & s(sx) \\ & = & (x\otr z)\otr(y\ud z),
\end{eqnarray*}\begin{eqnarray*}
(x\utr y)\utr(z\od y) 
& = & t(tx+(s-t)y)+(s-t)((s+t-a)z+(a-t)y) \\
&= & t^2x+(t(s-t)+(s-t)(a-t))y+(s-t)(s+t-a)z \\
& = & t(tx+(s-t)z)+(s-t)(ay+(s-a)z)\\ & = &(x\utr z)\utr(y\ud z),
\end{eqnarray*}\begin{eqnarray*}
(x\otr y)\od(z\otr y) 
& = & b(sx)+(s-b)(sz) \\
& = & s(bx+(s-b)z)\\ &=&(x\od z)\otr(y\utr z),
\end{eqnarray*}\begin{eqnarray*}
(x\utr y)\ud(z\utr y) 
& = &  a(tx+(s-t)y)+(s-a)(tz+(s-t)y) \\
& = &  t(ax+(s-a)z +(s-t)(sy) \\ 
& =& (x\ud z)\utr(y\otr z),
\end{eqnarray*}\begin{eqnarray*}
(x\otr y)\ud(z\otr y) 
& = & a(sx)+(s-a)(sz) \\
&= & s(ax+(s-a)z) \\ & = & (x\ud z)\otr(y\utr z) \quad\mathrm{and}
\end{eqnarray*}\begin{eqnarray*}
(x\utr y)\od(z\utr y) 
& = & b(tx+(s-t)y)+(s-b)(tz+(s-t)y)\\
& = & t(bx+(s-b)z)+(s-t)(sy)\\
& = & (x\od z)\utr(y\otr z)\\
\end{eqnarray*}
as required.
\end{proof}

\begin{example}
We can define finite psyquandles by selecting units $s,t,a,b\in\mathbb{Z}_n$
such that $s+t=a+b$. For instance, in $\mathbb{Z}_5$ we can select
$s=2$, $t=3$, $a=4$ and $b=1$; then $s+t=2+3=0=1+4$ and we have an Alexander 
psyquandle with operations
\[
\begin{array}{rcl}
x\utr y & = & 3x+4y \\
x\otr y & = & 2x
\end{array}\quad
\begin{array}{rcl}
x\ud y & = & 4x+4y \\
x\od y & = & x+y
\end{array}
\]
and operation matrix
\[
\left[\begin{array}{rrrrr|rrrrr|rrrrr|rrrrr}
2 & 1 & 5 & 4 & 3 & 2 & 2 & 2 & 2 & 2 & 3 & 2 & 1 & 5 & 4 & 2 & 3 & 4 & 5 & 1\\
5 & 4 & 3 & 2 & 1 & 4 & 4 & 4 & 4 & 4 & 2 & 1 & 5 & 4 & 3 & 3 & 4 & 5 & 1 & 2\\
3 & 2 & 1 & 5 & 4 & 1 & 1 & 1 & 1 & 1 & 1 & 5 & 4 & 3 & 2 & 4 & 5 & 1 & 2 & 3\\
1 & 5 & 4 & 3 & 2 & 3 & 3 & 3 & 3 & 3 & 5 & 4 & 3 & 2 & 1 & 5 & 1 & 2 & 3 & 4\\
4 & 3 & 2 & 1 & 5 & 5 & 5 & 5 & 5 & 5 & 4 & 3 & 2 & 1 & 5 & 1 & 2 & 3 & 4 & 5
\end{array}\right]
\]
where we use $5$ as the class of zero in $\mathbb{Z}_5$.
\end{example}

\begin{example}
If $X$ is a commutative ring with identity in which $2$ is invertible, we
can set $a=b=\frac{s+t}{2}$ to get pI-adequate psyquandle operations
\[x\ud y = \displaystyle \frac{s+t}{2}x+\frac{s-t}{2}y=x\od y.\]
We can interpret these operations as averaging the two possible
classical resolutions of an oriented precrossing. We call this type of 
psyquandle a \textit{Jablan psyquandle} since it was originally inspired by
Slavik Jablan's notion of precrossings as averages of two classical crossings.
For instance, in $X=\mathbb{Z}_5$ choosing $s=2$ and $t=4$ yields
\[x\ud y=3x+ y=x\od y.\]
\end{example}

\begin{example}
We can compute $\Phi_X^{\mathbb{Z}}$ for a singular link or pseudolink using
linear algebra when $X$ is an Alexander psyquandle. For example, the 
pseudoknot
\[\includegraphics{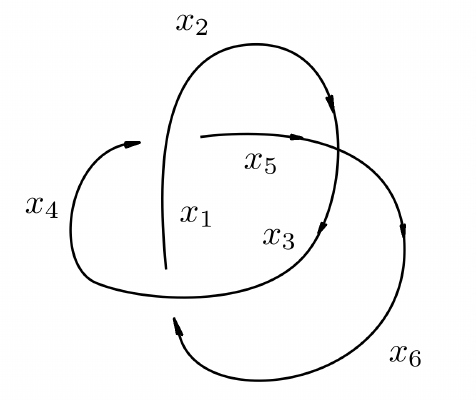}\]
has system of coloring equations given by 
\[\begin{array}{rcl}
sx_1 & = & x_2 \\
tx_4+(s-t) x_1 & = & x_5 \\
sx_3 & = & x_4 \\
tx_6 + (s-t) x_3 &=& x_1 \\
ax_5+(s-a)x_3 & =& x_6 \\
(s+t-a)x_3+(a-t) x_5 & = & x_2.
\end{array}\]
Choosing as a coloring psyquandle $X=\mathbb{Z}_5$ with $s=3$, $t=1$, $a=2$
and $b=2$,
this becomes
\[\begin{array}{rcl}
3x_1 +4x_2 & = & 0 \\
2x_1 +x_4+4x_5 & = & 0 \\
3x_3 +4x_4 & = & 0 \\
4x_1 +2x_3 +x_6 &=& 0 \\
x_3+ 2x_5 +4x_6 & =& 0 \\
4x_2+ 2x_3+ x_5 & = &0
\end{array}\]
which we can solve by row-reduction over $\mathbb{Z}_5:$
\[
\left[\begin{array}{rrrrrr}
3 & 4 & 0 & 0 & 0 & 0 \\ 
2 & 0 & 0 & 1 & 4 & 0 \\ 
0 & 0 & 3 & 4 & 0 & 0 \\ 
4 & 0 & 2 & 0 & 0 & 1 \\ 
0 & 0 & 1 & 0 & 2 & 4 \\ 
0 & 4 & 2 & 0 & 1 & 0 \\ 
\end{array}\right]
\rightarrow
\left[\begin{array}{rrrrrr}
1 & 4 & 0 & 4 & 1 & 0 \\ 
0 & 1 & 0 & 4 & 1 & 0 \\ 
0 & 0 & 1 & 0 & 2 & 4 \\ 
0 & 0 & 0 & 1 & 1 & 2 \\ 
0 & 0 & 0 & 0 & 1 & 0 \\ 
0 & 0 & 0 & 0 & 0 & 1 \\ 
\end{array}\right]
\]
so $\mathrm{dim} (\mathrm{ker}(A))=0$ and $\Phi_X^{\mathbb{Z}}(L)=1.$
Since the unknot has $\Phi_X^{\mathbb{Z}}(0_1)=5\ne 1$, this invariant
detects the (pseudo)knottedness of $L$.
\end{example}

Let $X=\mathbb{Z}_p$ and set $s=1$ and $t=-1$ so we have
\[x\utr y = -x+(1-(-1))y=2y-x\quad\mathrm{and}\quad x\otr y= x.\]
Colorings of classical knots and links by this type of biquandle are known as 
\textit{$p$-colorings}. 
Let us denote by $X_p$ and $X_p'$ respectively the Alexander psyquandle
structures on $X$ with $s=1$, $t=-1$, $a=1$ and $b=-1$ and 
$s=1$, $t=-1$, $a=-1$ and $b=1$
respectively. Observe that $X_p$ satisfies $x\utr y=x\ud y$ and $x\otr y=x\od y$
while $X_p'$ satisfies $x\utr y=x\od y$ and $x\otr y=x\ud y$. In particular,
we have the following observation:

\begin{observation}
An $X_p$-coloring of a pseudolink diagram $D$ coincides with a $p$-coloring
of the positive resolution of $D$, while an $X_p'$-coloring coincides with
a $p$-coloring of the negative resolution of $D$.
\end{observation}

In \cite{HJ}, two notions of $p$-colorability of pseudolinks were introduced. 
More precisely, a pseudolink $L$ is \textit{$p$-colorable} if 
every resolution of $L$ is $p$-colorable. 
A \textit{strong $p$-coloring} is a $p$-coloring at classical crossings such 
that at every precrossing, all four semiarcs have the same color.

\begin{lemma} Let $p\in \mathbb{Z}$ be odd.
A coloring of a pseudolink diagram which is both an $X_p$-coloring and an
$X_p'$-coloring is a strong $p$-coloring.
\end{lemma}

\begin{proof}
At precrossings we have
\[\includegraphics{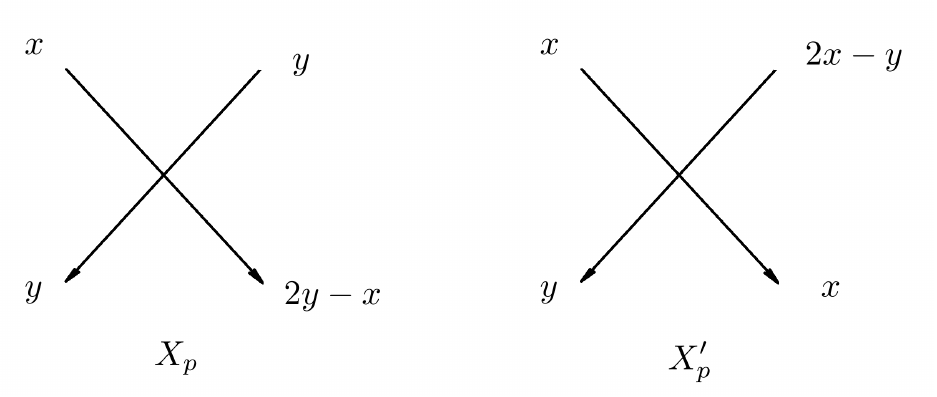}\]
so a coloring which satisfies both $X_p$ and $X_p'$ must satisfy 
$2x=2y$ at every precrossing. Since $p$ is odd, $2$ is invertible in 
$\mathbb{Z}_p$ and we have $x=y=2x-y=2y-x$ as required.
\end{proof}

\begin{corollary} Let $p\in\mathbb{Z}$ be odd. 
A pseudolink $L$ is strongly $p$-colorable if and only if
\[\mathrm{Hom}(\mathcal{P}(L),X_p))\cap\mathrm{Hom}(\mathcal{P}(L),X'_p))
\ne\emptyset.\]
\end{corollary}

Finally, we conclude with generalizations of the Alexander polynomial to the 
cases of singular links and pseudolinks.

Let $D$ be an oriented singular link diagram or pseudolink diagram. 
We obtain a homogeneous system of linear equations over $\Lambda'$ from the 
crossing relations of $D$, describing a presentation of the \textit{fundamental 
Alexander psyquandle} of $L$. In fact, using our crossing labelings this 
presentation is given by a matrix $A$ with entries in the polynomial ring 
$\hat\Lambda=\mathbb{Z}[t,s,a,b,t^{-1},s^{-1},a^{-1},b^{-1}]$ where 
$t^{-1}, s^{-1}, a^{-1},b^{-1}$ are independent variables and which has 
$\Lambda'=\hat\Lambda/(tt^{-1}-1,ss^{-1}-1,aa^{-1}-1,bb^{-1}-1,s+t-a-b)$ 
as a quotient. 
Following the same procedure described in \cite{CHN} (see also Chapter 6 in
\cite{L} for a nice summary of the classical case, and note that our matrix 
$A$ is the transpose of the analogous matrix in \cite{L}), we obtain a 
sequence of ideals $I_k\subset \hat\Lambda$ which are invariants of $L$ by 
setting $I_k$ to be the ideal in $\hat\Lambda$ generated by the codimension $k$ 
minors $M_k$ of $A$ together with the polynomials
$\{tt^{-1}-1,ss^{-1}-1,aa^{-1}-1,bb^{-1},s+t-a-b\}$.

\begin{definition}
Let $L$ be an oriented singular link or pseudolink. Any generator of the 
smallest principal ideal $P_k$ containing the ideal $I_k\subset \hat\Lambda$
generated by the codimension $k$ minors of a presentation matrix
$A$ and the polynomials $\{tt^{-1}-1,ss^{-1}-1,aa^{-1}-1,bb^{-1},s+t-a-b\}$
is the \textit{$k$th Alexander psyquandle polynomial} of $L$, and fixing
a monomial ordering $\prec$ on $\{t,s,a,b,t^{-1},s^{-1},a^{-1},b^{-1}\}$,
the reduced Gr\"obner basis for $I_k$ is the \textit{$k$th Alexander-Gr\"obner}
invariant of $L$.
\end{definition}
 
A useful special case is to use the Jablan psyquandle, i.e. set 
$a=b=\frac{s+t}{2}$ with coefficients in $\mathbb{Z}[\frac{1}{2}]$. More
precisely, we have:

\begin{definition}
The \textit{Jablan Polynomial} $\Delta_J(L)$ of an oriented pseudolink or 
singular link $L$ is any generator of the smallest principal ideal in 
$\Lambda_J=\mathbb{Z}[\frac{1}{2},s^{\pm 1}, t^{\pm 1},\frac{1}{s+t}]$ containing
the ideal generated by the codimension 1 minors of the Jablan psyquandle 
matrix of $L$ with $a=b=\frac{s+t}{2}$. 
\end{definition}

As in the case of the Alexander polynomial, the codimension 1 elementary 
ideal in the Jablan module is principal, so we can simply take any codimension
1 minor to compute $\Delta_J$ up to units. First, we have

\begin{lemma}
Let $L$ be a classical link considered as a pseudolink without precrossings.
Then $\Delta_J(L)$ is a homogeneous polynomial in $s$ and $t$ which specializes
to the Alexander polynomial up to powers of 2 by setting $s=1$.
\end{lemma}

\begin{proof}
The Jablan matrix of a classical link is equivalent by row and column moves 
to the block matrix 
\[\left[\begin{array}{c|c}
A' & 0\\ \hline
0 & I
\end{array}\right]\]
where $A'$ is the matrix obtained from the presentation matrix $A$ of the
Alexander quandle of $L$ by replacing every 1 with $s$. Then the codimension
1 minors of $J$ equal the codimension 1 minors of $A'$; these are homogeneous
since every entry is either $\pm s$, $t$ or $s-t$.
\end{proof}

We have the following standard lemma, sometimes given as an exercise in
commutative algebra courses:

\begin{lemma}\label{lp}
Let $R$ be a commutative ring with identity. Then the units in $R[x^{\pm 1}]$
have the form $rx^n$ where $r$ is a unit in $R$.
\end{lemma}

\begin{proof}
Any Laurent polynomial $p(x)=\sum_{k=a}^br_kx^k$ can be rewritten as
\[p(x)=x^a\sum_{k=a}^br_kx^{k-a}=x^aq(x) \]
where $q(0)=r_a\ne 0$. Then if $p(x)$ is a unit with inverse $p'(x)=x^{a'}q'(x)$
where $q'(0)=r_{a'}'\ne 0$, we have
\[1=pp'=qq'x^{a-a'}.\]
Evaluating at $x=0$ yields a contradiction unless $a=a'$, so we
have $pp'=qq'=1$; then $q$ is an invertible (non-Laurent) polynomial in $x$, 
that is to say, a unit in the ring $R$, and we have $p=rx^n$ as required.
\end{proof}

Applying the lemma \ref{lp} with $x=2,s,t$, we see that units in 
$\mathbb{Z}[2^{-1},s,t]$ are of the form $\pm2^js^kt^n$; then in the case of
adjoining $(s+t)^{-1}$, after factoring out the minimal power of $(s+t)$
and the minimal power of $s^jt^k$ in lexicographical ordering on $(j,k)$,
evaluation at $(0,0)$ yields the analogous result and we see that that 
the units in $\Lambda_J$ are of the form  $\pm 2^is^jt^k(s+t)^l$. Hence,
we can normalize a Jablan polynomial up to sign by clearing 
the denominator and canceling any common factors of $2,s,t$ and $(s+t)$. 

\begin{example}
The 2-bouquet graph $1_1^l$ in example \ref{ex:11l} has Jablan psyquandle matrix
\[\left[\begin{array}{cccc}
\frac{s + t}{2} & \frac{s- t}{2} & -1 & 0\\ 
\frac{s - t}{2} & \frac{s + t}{2} & 0 & -1 \\ 
s - t & t & 0 & -1 \\
s & 0 & -1 & 0
\end{array}
\right]\]
which has codimension 1 minors 
\[\left\{-\frac{s-t}{2},\ \frac{s-t}{2},\ \frac{-s(s-t)}{2},\ 
\frac{s(s-t)}{2}\right\}\] 
which have gcd $s-t$ up to units in $\Lambda_J$ (indeed, are equal up to 
units in $\Lambda_J$), so we have $\Delta_J(1_1^l)=s-t$. 
\end{example}

\begin{example}\label{ex:jp}
We computed the Jablan polynomials of a choice of orientation for
each of the pseudoknots and 2-bouquet
graphs in \cite{HHJJMR} and \cite{O} respectively. The results are collected 
in the tables.
\[
\begin{array}{r|l}
\Delta_J(L)& L \\\hline
1 & 3_1.1, 3_1.2,  4_1.1, 4_1.2, 4_1.3, \\
 & 5_1.1,5_1.2, 5_2.1, 5_2.2, 5_2.6, 5_2.9 \\
s^2+t^2 & 3_1.3,5_2.3, 5_2.4 \\
s^2-st+t^2 & 5_2.5, 5_2.10 \\
s^2-4st+t^2& 4_1.5 \\
s^2-6st+t^2& 4_1.4 \\
3s^2-2st+3t^2 & 5_2.7 \\
3s^2-4st+3t^2 & 5_2.8 \\
s^4+2s^3t+2s^2t^2+2st^3+t^4 & 5_1.3 \\
s^4+s^3t+st^3+t^4 & 5_1.4 \\
s^4+t^4 & 5_1.5\\
\end{array}
\begin{array}{r|l}\Delta_J(L)& L \\\hline
         1 & 0_1^k \\
      s^2+t^2 & 2_1^k\\
s^2-4st+t^2 & 3_1^k\\
s^2-st+t^2 & 4_2^k \\
s^2-3st+t^2 & 5_1^k \\
2s^2-5st+2t^2 & 6_{19}^k \\
3s^2-4st+3t^2 & 6_2^k\\
3s^2-5st+3t^2 & 6_6^k \\
3s^2-8st+3t^2 & 5_2^k \\
5s^2-8st+5t^2 & 6_3^k,\ 6_7^k\\
s^4+t^4 & 4_1^k, 4_3^k\\
s^4+s^3t-2s^2t^2+st^3+t^4 & 5_8^k \\
s^4-s^3t+s^2t^2-st^3+t^4 & 6_{10}^k, 6_{13}^k \\
s^4-s^3t-2s^2t^2-st^3+t^4 & 6_{18}^k \\
s^4-4s^3t+4s^2t^2-4st^3+t^4 & 5_3^k \\
\end{array}\]\[\begin{array}{r|l}\Delta_J(L)& L \\\hline
s^4-3s^3t+2s^2t^2-3st^3+t^4 & 5_4^k \\
s^4-2s^3t-2st^3+t^4 & 5_5^k\\
s^4-2s^3t+4s^2t^2-2st^3+t^4 & 5_6^k \\
s^4-3s^3t+6s^2t^2-3st^3+t^4 & 5_7^k \\
s^4-4s^3t+8s^2t^2-4st^3+t^4 & 6_{16}^k \\
s^4-5s^3t+6s^2t^2-5st^3+t^4 & 6_{14}^k \\
s^4-5s^3t+10s^2t^2-5st^3+t^4 & 6_{17}^k \\
s^4-6s^3t+8s^2t^2-6st^3+t^4 & 6_{11}^k \\
s^4-6s^3t+12s^2t^2-6st^3+t^4 & 6_{15}^k \\
s^4-7s^3t+10s^2t^2-7st^3+t^4 & 6_{12}^k \\
2s^4-s^3t-st^3+2t^4 & 6_5^k \\
2s^4-3s^3t+4s^2t^2-3st^3+2t^4 & 6_8^k \\
3s^4-4s^3t+4s^2t^2-4st^3+3t^4 & 6_4^k \\
3s^4-5s^3t+6s^2t^2-5st^3+3t^4 & 6_9^k \\
s^6+t^6 & 6_1^k \\
\end{array}
\begin{array}{r|l}
\Delta_J(L)& L \\\hline
s-t & 1_1^l, 6_{11}^l \\
5s-5t & 5_1^l \\
s^3-t^3 & 3_1^l \\
s^3-2s^2t+2st^2-t^3 & 4_1^l,5_3^l \\
s^3-4s^2t+4st^2-t^3 & 6_2^l\\
s^3-8s^2t+8st^2-t^3 & 6_{12}^l \\
2s^3-s^2t+st^2-2t^3 & 5_2^l\\
2s^3-3s^2t+3st^2-2t^3 & 6_1^l\\
2s^3-5s^2t+5st^2-2t^3 & 6_6^l \\
s^5-2s^3t^2+2s^2t^3-t^5 & 6_{10}^l\\
s^5-2s^4t+2s^3t^2-2s^2t^3+2st^4-t^5 & 6_5^l \\
s^5-2s^4t+4s^3t^2-4s^2t^3+2st^4-t^5 & 6_3^l,6_4^l \\
s^5-3s^4t+5s^3t^2-5s^2t^3+3st^4-t^5 & 6_4^l \\
s^5-4s^4t+6s^3t^2-6s^2t^3+4st^4-t^5 & 6_7^l \\
s^5-4s^4t+8s^3t^2-8s^2t^3+4st^4-t^5 & 6_9^l \\
\end{array}\]
\end{example}

In light of example \ref{ex:jp}, we make a few observations
in the following remarks:


\begin{remark}
The polynomials in Example \ref{ex:jp} are all homogeneous as previously noted 
and symmetric in the sense that the coefficients of $s^{n-k}t^k$ and $s^kt^{n-k}$ 
are equal. In the case of classical knots and links, symmetry in $t$ and $s$
follows from the fact that the upper and lower biquandles are isomorphic
and in our notation, the resulting polynomials are related by switching $s$ and
$t$. 
\end{remark}

\begin{remark}\label{rem3}
One alternative idea for an Alexander-style polynomial for a 
pseudoknot or pseudolink would be to take a weighted average of Alexander 
polynomials of the classical resolutions of the pseudoknot or pseudolink
with weights from the WeRe set. Indeed, at the level of Jablan matrix this
is effectively what we are doing.

However, it is not clear in general how to take a weighted average of
Alexander polynomials since 
the Alexander polynomial is only defined up to multiplication by units: 
should an average of $t$ and $t$ be 
$\frac{t+t}{2}=t$
or
$\frac{t+(-1)t}{2}=0$
or even
$\frac{(t^{-1})t+(t)t}{2}=\frac{1+t^2}{2}$?
We observe that in the cases above, the $s=1$ specialization of the Jablan 
polynomial of a pseudolink does in fact agree with a weighted sum of 
some choice of normalizations of Alexander 
polynomials of the classical resolutions: for example, pseudoknot $3_1.1$ has
Jablan polynomial $s^2+2st+t^2$, specializing to $1+2t+t^2$. If we 
symmetrize this in $t$, we obtain $t^{-1}+2+t$. Then $3_1.1$ has 
WeRe set 
\[\left\{\left(0_1,\frac{3}{4}\right),\left(3_1,\frac{1}{4}\right)\right\}\]
and taking a weighted sum of  symmetric normalizations with positive
leading coefficient of the Alexander polynomials of $0_1$ and $3_1$ and clearing
the denominator, we have 
\[3(1)+1(t^{-1}-1+t)=t^{-1}+2+t.\]

However, the pseudoknot $4_1.4$ has Jablan polynomial 
$\Delta_J(4_1.4)=s^2-6st+t^2$ and WeRe set 
\[\left\{\left(0_1,\frac{3}{4}\right),\left(4_1\frac{1}{4}\right)\right\}\]
with positive symmetric normalized Alexander polynomials
\[\Delta(0_1)=1\quad \Delta(4_1)=t^{-1}-3+t;\]
taking the weighted sum and clearing the denominator, we have
\[3\Delta(0_1)+1\Delta(4_1)=3+(t^{-1}-3+t)=t^{-1}+t\ne t^{-1}-6+t.\]
But, if we multiply the first polynomial by the unit $-1$, we obtain
\[3(-1)\Delta(0_1)+1\Delta(4_1)=-3+(t^{-1}-3+t)= t^{-1}-6+t,\]
coinciding with the specialization of $\Delta_J(4_1.4)$ as desired.
\end{remark}

In light of these remarks, we propose the following conjecture:

\begin{conjecture}\label{conj1}
There exists a choice of normalization rule for the Jablan polynomial
such that for every pseudoknot $K$ with WeRe set 
$S=\{(\alpha_1,K_1),\dots, (\alpha_n,K_n)\}$ we have
\[\Delta_J(K)=\sum_{j=1}^n \alpha_j\Delta_J(K_j).\]
\end{conjecture}



\section{Questions}\label{Q}

We conclude with some questions for future research.

The main question, of course, is conjecture \ref{conj1} true? More precisely,
what normalization rule makes 
\[\Delta_J(K)=\sum_{j=1}^n \alpha_j\Delta_J(K_j)\]
for pseudoknots with WeRe set $S=\{(\alpha_1,K_1),\dots, (\alpha_n,K_n)\}$?

What enhancements of psyquandle counting invariants can be defined? 
Enhancements of psyquandle counting invariants will be the topics of 
future papers.

\bibliography{sn-no-rs}{}
\bibliographystyle{abbrv}

\bigskip

\noindent
\textsc{Department of Mathematical Sciences \\
Claremont McKenna College \\
850 Columbia Ave. \\
Claremont, CA 91711} 

\

\noindent
\textsc{Department of Teacher Education \\
Shumei University \\
1-1 Daigaku-cho, Yachiyo \\
Chiba Prefecture 276-0003, Japan}

\

\noindent
\textsc{Department of Mathematics\\ North Carolina State University\\ 2311 Stinson dr. \\ Raleigh, NC 27695-8205
}

\end{document}